\documentclass[11pt, a4paper]{amsart}

\usepackage{amsmath} 
\usepackage{amssymb} 
\usepackage{amscd} 
\usepackage{tabularx} 
\usepackage[all,cmtip,line]{xy}
\usepackage{enumerate}
\usepackage{hyperref}
\usepackage{bm}

\usepackage{amsthm}
\newtheorem{thm}{\textbf{Theorem}}[section]
\newtheorem{defn}[thm]{\textbf{Definition}}
\newtheorem{prop}[thm]{\textbf{Proposition}}
\newtheorem{lem}[thm]{\textbf{Lemma}}
\newtheorem{corollary}[thm]{\text{Corollary}}
\newtheorem{rem}[thm]{\textbf{Remark}}

\newtheorem*{thmn}{\textbf{Theorem}}
\newtheorem*{Defn}{\textbf{Definition}}

\def\Q{\mathbb{Q}}
\def\Z{\mathbb{Z}}
\def\C{\mathbb{C}}
\def\A{\mathbb{A}}

\def\Gal{\operatorname{Gal}}

\def\GL{\operatorname{GL}}
\def\SL{\operatorname{SL}}
\def\PGL{\operatorname{PGL}}

\def\til{\widetilde}

\def\Frob{\operatorname{Frob}}

\def\D{\mathcal{D}}
\def\Qbar{\overline{\mathbb{Q}}}

\newcommand{\Hom}{\operatorname{Hom}}

\newcommand{\cH}{\mathcal{H}}

\title{$L$-Indistinguishability on Eigenvarieties}
\author{Judith Ludwig}
\email{jludwig@math.uni-bonn.de}

\begin{document}
\date{\today}
\maketitle
\begin{abstract}
In this article, we construct examples of $L$-indistinguishable overconvergent eigenforms for an inner form of $\SL_2$.
\end{abstract}
\section{Introduction}
The purpose of this article is to show the existence of points on eigenvarieties for an inner form of $\SL_2$, whose associated systems of Hecke eigenvalues come from classical automorphic representations of an inner form of $\GL_2$, but which are \textit{not classical} themselves, in the sense that there are no classical forms for these systems of Hecke eigenvalues in the corresponding spaces of overconvergent forms. By construction there are \textit{overconvergent} forms giving rise to these points. For each such point we also construct a twin on a different eigenvariety which is classical. The overconvergent non-classical form $f$ on the first eigenvariety and the classical form $g$ on the second eigenvariety are $L$-indistinguishable in the sense that they give rise to the same Galois representation. 
Although there is no definition of a global $p$-adic $L$-packet, our results suggest that, for any future definition, $f$ and $g$ should lie in the same $L$-packet. 

We describe our results in more detail:
Let $B/\Q$ be a definite quaternion algebra and denote by $S_B$ the set of primes where $B$ ramifies. 
Let $\til{G}$ be the algebraic group over $\Q$ defined by the units $B^*$ and $G$ the subgroup of elements of reduced norm one. Fix a prime $p \notin S_B$ and a finite extension $E/\Q_p$. 

For $S$ a finite set of places which includes $p$ and $S_B$, we have a Hecke algebra $\til{\cH}_S:=\til{\cH}_{ur,S} \otimes_E \til{\mathcal{A}}_p$ for $\til{G}$, which is the product of the spherical Hecke algebras at all places not in $S$ and an Atkin-Lehner algebra at $p$, and an analogue $\cH_S$ for $G$. Any idempotent $\til{e}=\otimes \til{e}_l \in C^{\infty}_c(\til{G}(\A_f^p), \Qbar)$, such that $\til{e}_l=\mathbf{1}_{\GL_2(\Z_l)} $ for all $l \notin S$, gives rise to an eigenvariety $\D(\til{e})$ of idempotent type $\til{e}$, whose underlying set of points embeds 
\[
\D(\til{e})(\Qbar_p) \hookrightarrow \Hom(\til{\mathcal{H}}_S,\Qbar_p) \times \til{\mathcal{W}}(\Qbar_p),
\]
where $\til{\mathcal{W}}= \Hom_{cts}((\Z_p^*)^2, \mathbb{G}_m)$ denotes the usual weight space.
If $e \in C^{\infty}_c(G(\A_f^p),\Qbar)$ is an idempotent with the same set $S$ of bad places, we have an eigenvariety $\D(e)$ of idempotent type $e$ for $G$, whose underlying set of points embeds
\[
\D(e)(\Qbar_p) \hookrightarrow \Hom(\mathcal{H}_S,\Qbar_p)\times \mathcal{W}(\Qbar_p),
\]
for the corresponding weight space $\mathcal{W} = \Hom_{cts}(\Z_p^*,\mathbb{G}_m)$.
There are natural maps 
\[ \mathcal{H}_S \hookrightarrow \til{\mathcal{H}}_S , \  \til{\mathcal{W}} \rightarrow \mathcal{W}.\]
 
\begin{Defn}[Definition \ref{clpt}]
A point $z$ on an eigenvariety of idempotent type is called \emph{classical}, if there is a classical automorphic eigenform in the corresponding space of overconvergent forms, whose system of Hecke eigenvalues is that defined by~$z$. 
\end{Defn}

Let $\pi(\til{\theta})$ be an algebraic automorphic representation of $\til{G}(\A)$ associated to a Gr\"o{\ss}encharacter $\til{\theta}$ of an imaginary quadratic field $L$. Assume that $p$ splits in $L$. Then such a representation gives rise to two points on $\D(\til{e})$ for a suitable idempotent $\til{e} \in C^{\infty}_c(\til{G}(\A_f^p),E)$, one which is ordinary and one which is of critical slope. 

Let $\til{x}$ be the point of critical slope and consider its image in $ \Hom(\mathcal{H}_S,\Qbar_p) \times \mathcal{W}(\Qbar_p)$ under the composite of the maps

\[\xymatrix{
\D(\til{e})(\Qbar_p) \ar[r] \ar[rd]^\phi &  \ar[d] \Hom(\til{\mathcal{H}}_S,\Qbar_p) \times \til{\mathcal{W}}(\Qbar_p) \\
& \Hom(\mathcal{H}_S,\Qbar_p) \times \mathcal{W}(\Qbar_p) ,  }
\]
which we denote by $\phi$.   

Our main theorem is the following. 
\begin{thmn}[Theorem \ref{mainthm}] There exist automorphic representations $\pi(\til{\theta})$ of $\til{G}(\A)$ as above together with idempotents $\til{e} \in C^{\infty}_c(\til{G}(\A_f^p),\Qbar)$ and $e_1, e_2 \in C^{\infty}_c(G(\A_f^p),\Qbar)$, such that, using the above notation, the image $\phi(\til{x})$ of the critical slope refinement~$\til{x}$ of $\pi(\til{\theta})$ lifts to a non-classical point on the eigenvariety $\D(e_1)$ and to a classical point on $\D(e_2)$. 
\end{thmn}

The proof of the theorem uses a $p$-adic version of a Labesse--Langlands transfer proved by the author in \cite{p-adicLL}. 
The representations $\pi(\til{\theta})$ and the idempotents~$e_1 $ and~$e_2$ are constructed in such a way, that each $e_i$ \textit{sees} exactly one member of the $L$-packet $\Pi(\pi(\til{\theta}))$, i.e., for $i=1,2$ there exists a unique element $\pi_i \in \Pi(\pi(\til{\theta}))$ such that $e_i (\pi_i)_f^p \neq 0$. Moreover $m(\pi_1)$, the multiplicity of $\pi_1$ in the automorphic spectrum of~$G$, is zero and $\pi_2$ is automorphic. In particular, this implies that $\phi(\til{x})$ lifts to $\D(e_2)$. In order to show that it also lifts to $\D(e_1)$, the crucial point is that in a neighbourhood of $\til{x}$ we can find many points associated to automorphic representations that do not come from a Gr\"o{\ss}encharacter. These automorphic representations of~$\til{G}(\A)$ give rise to stable $L$-packets of $G$ and therefore their images under $\phi$ all lift to~$\D(e_1)$.

By construction, points on eigenvarieties give rise to systems of Hecke eigenvalues occurring in spaces of overconvergent automorphic forms. So in particular the theorem shows the existence of an overconvergent eigenform $f$ of tame level $e_1$, whose system of Hecke eigenvalues $\psi_f$ comes from a classical automorphic representation of $\til{G}(\A)$. The multiplicity formulae of Labesse and Langlands rule out that there is a classical eigenform for $\psi_f$ of the same tame level. 
On the other hand the point on $\D(e_2)$ comes from a classical form, say $g$. The associated Galois representations $\rho_f$ and $\rho_g$ agree as the systems of $\cH_S$-eigenvalues for $f$ and $g$ are the same. In this sense the two forms are $L$-indistinguishable. 
\vspace{3mm}

\textbf{Notation}. Fix embeddings $\iota_{\infty}:\Qbar \hookrightarrow \C$ and $\iota_p:\Qbar \hookrightarrow \Qbar_p$ as well as a finite extension $E/\Q_p$ with ring of integers $\mathcal{O}_E$ and an embedding $E \subset \Qbar_p$. 
For a number field $K$ we denote by $G_K:=\Gal(\overline{K}/K)$ the absolute Galois group of an algebraic closure of $K$. For a finite set $S$ of places of $K$, let $G_{K,S}$ be the Galois group of a maximal extension of $K$ that is unramified outside $S$, and for $v \notin S$, let $\operatorname{Frob}_v \in G_{K,S}$ denote a representative of the geometric Frobenius at $v$.  

We will frequently choose idempotents $e \in C^{\infty}_c(H(\A_f^p),\Qbar)$ where $H$ is equal to either $\til{G}$ or to $G$. We will always assume that $E$ is chosen big enough so that $\iota_p \circ e$ takes values in $E$. We ease notation by dropping the embeddings from the notation when it is obvious, e.g., for a complex representation $\pi_f^p$ of $H(\A_f^p)$ we write $e \cdot \pi_f^p$ instead of $(\iota_\infty \circ e ) \cdot \pi_f^p$. 
Furthermore we assume that the idempotents we consider are given as a tensor product of local idempotents $e_l \in C^{\infty}_c(H(\Q_l),\Qbar)$, where $e_l = \mathbf{1}_{H(\Z_l)}$ for almost all $l$. We denote by $S(e)$ the minimal finite set of finite primes containing~$S_B$ and~$p$, such that  
$e_l = \mathbf{1}_{H(\Z_l)}$ for all $l \notin S(e)$. 

For a rigid analytic space $X$ defined over $E$, any point is defined over a finite extension of $E$. We write $X(\Qbar_p):= \bigcup_{E'/E \text{ finite}} X(E')$. 
\vspace{3mm}

\textbf{Acknowledgements}. The author would like to thank Ga{\"e}tan Chenevier, Eugen Hellmann, James Newton, Peter Scholze and Benjamin Schraen for helpful conversations and Kevin Buzzard for his comments on an earlier draft of this paper. The author would also like to thank the anonymous referee for various helpful comments and suggestions. The author was supported by the SFB/TR 45 of the DFG.

\section{Eigenvarieties and a $p$-adic Labesse--Langlands transfer}
Let $B, \til{G}$ and $G$ be as in Section 1 and let $e:=\otimes e_l \in C^\infty_c(G(\A_f^p), \Qbar)$ be an idempotent. 
Choose $S \supset S(e)$ and define
$$ \mathcal{H}_{ur,S}:= \bigotimes_{l\notin S} { '} \mathcal{H}_E(\SL_2(\Q_l),\SL_2(\Z_l)),$$
where $\mathcal{H}_E(\SL_2(\Q_l),\SL_2(\Z_l))$ is the algebra under convolution of compactly supported $E$-valued functions on $\SL_2(\Q_l)$ that are bi-invariant under $\SL_2(\Z_l)$. 
Denote by $I$ the Iwahori subgroup of $\SL_2(\Q_p)$ given by 
\[I:= \left\{ \left(\begin{array}{cc} a & b \\ c& d \end{array}\right) \in \SL_2(\Z_p) : c \equiv 0 \mod p \right\}.\]
Let 
\[
\mathcal{H}_S:=\mathcal{H}_{ur,S} \otimes_E \mathcal{A}_p
\]
be the Hecke algebra, where $\mathcal{A}_p$ is the commutative $E$-subalgebra of the Iwahori Hecke algebra $\mathcal{H}_E(\SL_2(\Q_p),I)$ generated by the characteristic function on the double coset
\[ \operatorname{I} \left(\begin{array}{cc} p^{-1} &  \\  & p \end{array}\right) \operatorname{I}. 
\] 
We let $\mathcal{W}:=\Hom_{cts}(\Z_p^*,\mathbb{G}_m)$ be the usual weight space. 
Using the spaces of overconvergent forms for $G$ as constructed in \cite{david} and Buzzard's machine, one can attach to this data an eigenvariety $\D(e,S)$ of idempotent type $e$ (cf.\ \cite{kevin}, \cite{david}). 

To be precise, one also has to make a choice of a compact operator in the construction, which we fix once and for all to be 
\[ u_0:= \mathbf{1}_{\cH_{ur,S}} \otimes \mathbf{1}_{\operatorname{I} \left(^{p^{-1}} \ _p \right) \operatorname{I}} \in \cH_S.
\]

The eigenvariety $\D(e,S)$ is a rigid analytic space defined over $E$ and it comes equipped with a locally finite (on the source) morphism 
\[
\omega: \D(e,S)\rightarrow \mathcal{W}
\] 
and an $E$-algebra homomorphism 
\[\psi:\mathcal{H}_S \rightarrow \mathcal{O}(\D(e,S)).\]
The points of $\D(e,S)$ correspond to finite slope systems of Hecke eigenvalues occurring in the space of overconvergent forms mentioned above. Moreover the map 
\begin{eqnarray*}
\D(e,S)(\Qbar_p) & \rightarrow & \operatorname{Hom}(\mathcal{H}_S,\Qbar_p) \times \mathcal{W}(\Qbar_p)\\
x & \mapsto & (\psi_x(h):= \psi(h)(x), \omega(x))
\end{eqnarray*}
is an injection (cf.\ Lemma 7.2.7 of \cite{BC}). 

Likewise, starting from an idempotent $\til{e} \in C^{\infty}_c(\til{G}(\A_f^p),\Qbar)$, a set $S \supset S(\til{e})$ and an associated Hecke algebra $\til{\mathcal{H}}_S=\til{\mathcal{H}}_{ur,S}\otimes_E \til{\mathcal{A}}_p$, we can build an eigenvariety $\D(\til{e},S)$ of idempotent type $\til{e}$ for $\til{G}$. The weight space in this case is $\til{\mathcal{W}}:= \Hom_{cts}((\Z_p^*)^2, \mathbb{G}_m)$. We have natural maps (see \cite[Sections 2.2 and 2.3]{p-adicLL}) 
\[
\mu:\til{\mathcal{W}} \longrightarrow \mathcal{W},
\]
\[
\lambda:\mathcal{H}_S \hookrightarrow \til{\mathcal{H}}_S.
\]
In the construction of $\D(\til{e},S)$ we always choose $\lambda(u_0)$ as the compact operator.

There are natural $\Q$-structures on the Hecke algebras defined by the subalgebras of $\Q$-valued functions which we denote by $\cH_{\Q,S}, \cH_{\Q,ur,S}$ etc. The monomorphism $\lambda$ can in fact be defined over $\Q$ (see Section 2 of \cite{p-adicLL} for details). In particular we constructed in Lemma 2.10 of \cite{p-adicLL} an inclusion of $\Q$-algebras
\[\lambda_{l,\Q}:\mathcal{H}_{\Q}(\SL_2(\Q_l),\SL_2(\Z_l)) \hookrightarrow \mathcal{H}_{\Q}(\GL_2(\Q_l),\GL_2(\Z_l)).\]

We also refer to \cite{p-adicLL} for details regarding the construction as well as general properties of the eigenvarieties. 

\begin{rem} There exists a Zariski-dense and accumulation subset 
\[Z\subset \D(e,S)(\Qbar_p)
\]
 coming from $p$-refined classical automorphic representations as defined in \cite{p-adicLL}, Def.\ 3.14. This is proved as usual (cf. \cite[Section 6.4.5]{chenevier} and also \cite[Prop. 3.9]{p-adicLL} for a proof of the analogous assertion for $\til{G}$ in the same notation) using the fact that forms of small slope are classical and that classical weights are dense in weight space.
\end{rem}

Define $t_l \in \mathcal{H}_{ur,S}$, as the characteristic function on the double coset
\[
\SL_2(\widehat{\Z}^S) \left(\begin{array}{cc} l & 0\\0 & l^{-1}\end{array}\right) \SL_2(\widehat{\Z}^S),
\]
where $\left(\begin{smallmatrix} l& \\ & l^{-1} \end{smallmatrix}\right)$ is understood to be the matrix in $\SL_2(\widehat{\Z}^S)=\prod_{q\notin S} \SL_2(\Z_q)$ which is equal to $1$ for all $q \neq l$ and equal to $\left(\begin{smallmatrix} l& \\ & l^{-1} \end{smallmatrix}\right)$ at $l$. 
Furthermore let  
\[h_l:= \frac{1}{l}\left( t_l + 1 \right) \in \cH_{ur,S}.
\]
Note that $h_l$ is an element of the subalgebra $\cH^0_{ur,S}\subset \cH_{ur,S}$ of $\mathcal{O}_E$-valued functions. 

\begin{lem}\label{galoisreps} Let $\D(e,S)$ be an eigenvariety of idempotent type for $G$. Then there exists a 3-dimensional pseudo-representation
\[ T: G_{\Q,S} \rightarrow \mathcal{O}(\D(e,S))\]
 such that $T(\Frob_l)=\psi(h_l)$ for all $l\notin S$. 
\end{lem}
\begin{proof} By the previous remark we have a Zariski dense subset $Z\subset \D(e,S)(\Qbar_p)$ of classical points. A point $z \in Z$ comes from an algebraic automorphic representation $\pi$ of $G(\A)$ and there is a projective Galois representation 
\[
\rho_z: G_{\Q,S} \rightarrow \PGL_2(\Qbar_p)
\] 
associated to $\pi$. Namely, if $\widetilde{\pi}$ is any algebraic automorphic representation of $\til{G}(\A)$ which is unramified outside $S$ and lifts $\pi$, let $\rho(\til{\pi}): G_{\Q,S}\rightarrow \GL_2(\Qbar_p)$ be the attached Galois representation by Deligne. It has the property that for any $l \notin S$, the characteristic polynomial of $\rho(\til{\pi})(\operatorname{Frob}_l)$ is given by $X^2 - T_l(\til{\pi}) X +lS_l(\til{\pi})$,  
where $T_l(\til{\pi}):= \iota_p \circ \iota^{-1}_{\infty}(\mu_l)$ and $\mu_l$ is the eigenvalue of 
\[
T_l := \mathbf{1}_{\GL_2(\widehat{\Z}^S) \left(\begin{smallmatrix} l& \\ & 1 \end{smallmatrix}\right)\GL_2(\widehat{\Z}^S)} \in \til{\cH}_{\Q,ur,S}
\] 
on $(\til{\pi}_f^S)^{\GL_2(\widehat{\Z}^S)}$ and similarly for $S_l = \mathbf{1}_{\GL_2(\widehat{\Z}^S) \left(\begin{smallmatrix} l& \\ & l \end{smallmatrix}\right)\GL_2(\widehat{\Z}^S)} \in \til{\cH}_{\Q,ur,S}$. 

Then 
\[
\rho_z = \eta \circ \rho (\til{\pi}): G_{\Q,S}\rightarrow \PGL_2(\Qbar_p)
\]
is the composition of $\rho(\til{\pi})$ and the natural homomorphism $\eta: \GL_2(\Qbar_p) \rightarrow \PGL_2(\Qbar_p)$. There is a monomorphism $\iota:\PGL_2(\Qbar_p) \hookrightarrow \GL_3(\Qbar_p)$ coming from the adjoint action, which identifies $\PGL_2(\Qbar_p)$ with $\operatorname{SO}_3(\Qbar_p)$. Define 
\[\sigma_z:=\iota \circ \rho_{z}: G_{\Q,S}\rightarrow \GL_3(\Qbar_p).\]
Then 
\[
\sigma_z \cong  \operatorname{Sym}^2(\rho(\til{\pi})) \otimes \operatorname{det}(\rho(\til{\pi}))^{-1} 
\]
and an easy calculation shows that for all $l \notin S$ 
\begin{eqnarray*} 
\operatorname{Tr}(\sigma_z(\Frob_l)) &=& \operatorname{Tr}^2(\rho(\til{\pi})(\Frob_l))/ \operatorname{det}(\rho(\til{\pi})(\Frob_l)) - 1 \\
&=& T^2_l(\til{\pi})/(l S_l(\til{\pi})) -1.
\end{eqnarray*}

Now an elementary calculation shows that $T^2_l/(lS_l) -1 = \lambda_\Q(h_l)$ and therefore
\[T^2_l(\til{\pi})/(l S_l(\til{\pi})) -1 = \psi(h_l)(z). 
\] 

The lemma now follows from Proposition 7.1.1 of \cite{chenevier} (see also Section 3.1.3 of \cite{p-adicLL}), i.e., Hypothesis $\mathbf{H}$ in \cite{chenevier} is satisfied using the Zariski-dense set $Z \subset \D(e,S)(\Qbar_p)$, the representations $\sigma_z$ for $z \in Z$ and the family of functions $\psi(h_l) \in \mathcal{O}(\D(e,S))$.   
\end{proof}

\begin{rem}
In the following we abbreviate $\D(e):=\D(e,S(e))$ and $\D(\til{e}):=\D(\til{e},S(\til{e}))$. 
\end{rem}
Below we will often use so-called \textit{special idempotents} attached to a finite set of Bernstein components. Let $F/ \Q_l$ be a finite extension, $H(F)$ the $F$-points of a reductive group over $F$. Given a Bernstein component $\mathfrak{s}$ of the category of smooth $\Qbar$-representations of $H(F)$, there is an idempotent ~$e_{\mathfrak{s}} \in C^{\infty}_c(H(F), \Qbar)$ such that for an irreducible smooth $\Qbar$-representation $\sigma$ of $H(F), $  ~$e_{\mathfrak{s}}\cdot\sigma\neq  0$ if and only if $\sigma$ is contained in the Bernstein component $\mathfrak{s}$. Similarly one can attach an idempotent to a finite set $\Sigma$ of Bernstein components.
We refer to Section~3 of \cite{bushnell} for a nice overview and to Proposition 3.13 of loc.cit.\ for the existence of these so-called special idempotents. 

When we choose special idempotents below, we may always assume they take values in $\Qbar$ as all automorphic representations we deal with in this paper are algebraic and have the property that their finite part is defined over~$\Qbar$.

We also want to remark here that irreducible supercuspidal representations that are in the same Bernstein component differ from each other by a twist by an unramified character. In particular, if two irreducible supercuspidal representations $\sigma$ and $\sigma'$ of $\SL_2(\Q_l)$ are in the same Bernstein component, then $\sigma\cong\sigma'$. 

We use the following notation: Let $\mathfrak{s}$ be the Bernstein component of $\til{G}(\Q_l)$ defined by a supercuspidal representation $\til{\pi}_l$ of $\til{G}(\Q_l)$. Then we denote by $\operatorname{Res}^{\til{G}}_G(\mathfrak{s})$ the finite set of Bernstein components defined by the representations occurring in $\til{\pi}_l|_{G(\Q_l)}$.
Note this is well defined. 

The classical transfer, by which we just mean the map that attaches to an automorphic representation $\til{\pi}$ of $\til{G}(\A)$ an $L$-packet of representations $\Pi(\til{\pi})$ of $G(\A)$, can be interpolated to maps between suitable eigenvarieties. 

\begin{defn} Two idempotents $\widetilde{e} \in C^\infty_c(\widetilde{G}(\A^p_f), \Qbar)$ and $e \in C^\infty_c(G(\A^p_f),\Qbar)$ are called \emph{Langlands compatible} if they satisfy: For any discrete automorphic representation $\til{\pi}$ of $\widetilde{G}(\A)$ with $\widetilde{e} \cdot \widetilde{\pi}^p_f \neq 0$ and any $\tau \in \Pi(\widetilde{\pi}_p)$, there exists an element $\pi$ in the packet $\Pi(\widetilde{\pi})$, such that 
\begin{itemize}
	\item $m(\pi) > 0$,
  \item  $e \cdot \pi^p_f \neq 0$ and
	\item  $\pi_p = \tau$.
\end{itemize}
\end{defn}

\begin{thm} \begin{enumerate}
	\item  Let $\til{e}\in C^{\infty}_c(\til{G}(\A_f^p),\Qbar)$ be an idempotent. Then there exists an idempotent $e \in C^\infty_c(G(\A^p_f),\Qbar)$ with $S(e)=S(\til{e})$ and such that
$\til{e}$ and $e$ are Langlands compatible. Define S:=S(e).
\item Assume $\til{e}= \otimes \til{e}_l\in C^{\infty}_c(\til{G}(\A_f^p),\Qbar)$ has the property that for all $l \in S(\til{e})$, $\til{e}_l$ is a special idempotent attached to a supercuspidal Bernstein component $\mathfrak{s_l}$. Define $e:=\otimes e_l\in C^\infty_c(G(\A^p_f),\Qbar)$ where $e_l:= e_{\SL_2(\Z_l)}$ for all $l \notin S(\til{e})$ and for $l \in S(\til{e})$, $e_l$ is a special idempotent attached to the finite set of Bernstein components $\operatorname{Res}^{\til{G}}_G(\mathfrak{s}_l)$. Then $\til{e}$ and $e$ are Langlands compatible. 
\item For any two Langlands compatible idempotents $\til{e}\in C^{\infty}_c(\til{G}(\A_f^p),\Qbar)$ and $e\in C^\infty_c(G(\A^p_f),\Qbar)$ with the same set $S$ of bad places there exists a morphism $\zeta:\D(\til{e}) \rightarrow \D(e)$ such that the diagrams 	
\[
\xymatrix{
\D(\til{e})  \ar[d]^{\widetilde{\omega}} \ar[r]^\zeta &\D(e) \ar[d]^{\omega} \\
\widetilde{\mathcal{W}} \ar[r]^\mu &\mathcal{W} }
	\hspace{1cm}
	\xymatrix{
\mathcal{H}_S \ \ar[d] \ar@{^{(}->}[r]^\lambda & \widetilde{\mathcal{H}}_S \ar[d]\\
  \mathcal{O}(\D(e)) \ar[r]^{\zeta^*}  &\mathcal{O}(\D(\til{e})) }  
\]
commute.	
\end{enumerate} 
\label{p-adictransfer}
\end{thm}
\begin{proof} Part (1) follows from Proposition 4.15 and Proposition 4.16 of \cite{p-adicLL}, once we remark that for any idempotent $\til{e} \in C^{\infty}_c(\til{G}(\A_f^p),\Qbar)$, there exists a compact open subgroup $\til{K}\subset \til{G}(\A_f^p)$ such that $ e_{\til{K}} \cdot \til{e}= \til{e} =\til{e} \cdot e_{\til{K}}$. 
Part (3) follows from \cite{p-adicLL} Theorem 5.7 and the proof of it.

Part (2) can be proved in the same way as Proposition 4.16 of \cite{p-adicLL}. Namely for an automorphic representation $\til{\pi}$ of $\til{G}(\A)$ with $\widetilde{e} \cdot \widetilde{\pi}^p_f \neq 0$ and any $\tau \in \Pi(\til{\pi}_p)$ define 
\begin{eqnarray*} Y(\til{\pi},\tau)&:=&\{\pi \in \Pi(\til{\pi}) | e\cdot(\pi^p_f) \neq 0, \pi_p=\tau\} \\
&=& \{\pi \in \Pi(\til{\pi}) | \pi_l = \pi_l^0 \ \forall l \notin S(\til{e}), \pi_p=\tau\},
\end{eqnarray*}
where in the last line $\pi_l^0$ denotes the unique member of the $L$-packet $\Pi(\til{\pi}_l)$ with $(\pi_l^0)^{\SL_2(\Z_l)}\neq 0$.
We need to show that there exists $\pi \in Y(\til{\pi},\tau)$ such that $m(\pi)>0$. For that let $\pi \in Y(\til{\pi},\tau)$ be arbitrary and assume $m(\pi)=0$. Then by Proposition 4.11 of \cite{p-adicLL} we may change $\pi$ at a prime $l \in S_B$ to a different representation in the local $L$-packet $\Pi(\til{\pi}_l)$ to get a representation $\pi'$ which is automorphic and still in $Y(\til{\pi},\tau)$.
\end{proof}
\begin{rem}\label{diagram} We recall that $\zeta$ is constructed using two auxiliary eigenvarieties $\D'(\til{e})$ and $\D''(e)$, which are described in Section 3.3 and 3.4 of \cite{p-adicLL}. $\D'(\til{e})$ is the eigenvariety which apart from the Hecke algebra is build from the same data as $\D(\widetilde{e})$ but where the Hecke algebra is replaced by $\mathcal{H}_S$. It comes equipped with a morphism $\omega':\D'(\til{e}) \rightarrow \til{\mathcal{W}}$ and the points of $\D'(\til{e})$ embed
\[ \D'(\til{e})(\Qbar_p) \hookrightarrow \Hom(\cH_S, \Qbar_p) \times \til{\mathcal{W}}(\Qbar_p).
\]
We have a morphism $\lambda':\D(\til{e}) \rightarrow \D'(\til{e})$, which on points is given by 
\[
(\psi_x,\til{\omega}(x)) \mapsto (\psi_x|_{\cH_S}, \til{\omega}(x)).
\]  

The second eigenvariety $\D''(e)$ is simply defined as the pullback $\til{\mathcal{W}} \times_{\mathcal{W}} \D(e)$ and the morphism $\zeta$ is the composite 
\[
\xymatrix{
\D(\til{e}) \ar[d]^{\til{\omega}} \ar[r]^{\lambda'} & \D'(\til{e}) \ar[d]^{\omega'} \ar[r]^{\xi} & \ar[d] \D''(e) \ar[r] &\D(e) \ar[d]^{\omega} \\
\widetilde{\mathcal{W}} \ar[r]^{\operatorname{id}} & \widetilde{\mathcal{W}} \ar[r]^{\operatorname{id}}  &\widetilde{\mathcal{W}}  \ar[r]^\mu &\mathcal{W} }.
\]
By Proposition 5.6 of \cite{p-adicLL}, the morphism $\xi$ is a closed immersion. This is important in what follows.  
\end{rem}

\begin{rem} \label{dividemp} If $e, e' \in C^{\infty}_c(G(\A_f^p),\Qbar)$ are two idempotents as above, such that $S(e)=S(e')$ and such that $e'_l|e_l$ for all $l \in S \backslash \{p\}$, then there exists a closed immersion $\D(e) \hookrightarrow \D(e')$. (cf.\ \cite{BC} Section 7.3).
\end{rem}

\section{Slopes of CM points}
We determine the slopes of points on eigenvarieties $\D(\til{e})$ that arise from automorphic representations $\pi(\til{\theta})$ of $\til{G}(\A)$ coming from a Gr{\"o}{\ss}encharacter. 

We view $\Z^2$ as a subset of weight space via
\[
\Z^2 \hookrightarrow \til{\mathcal{W}}(\Qbar_p), (k_1,k_2) \mapsto \left((z_1,z_2) \mapsto z_1^{k_1}z_2^{k_2} \right).
\]
Let $\underline{k}=(k_1,k_2) \in \Z^2$, $k_1 \geq k_2$. We denote by $\widetilde{I}$ the Iwahori subgroup of $\GL_2(\Q_p)$ given by
\[\widetilde{I}:= \left\{ \left(\begin{array}{cc} a & b \\ c& d \end{array}\right) \in \GL_2(\Z_p) : c \equiv 0 \mod p \right\}.\]

Recall (cf.\ Section 7.2.2 of \cite{BC} and Definition 3.14 of ~\cite{p-adicLL}) that a $p$-refined automorphic representation of weight $\underline{k}$ of $\widetilde{G}(\A)$ is a pair $(\til{\pi}, \chi)$ such that
\begin{itemize}
	\item $\til{\pi}$ is an automorphic representation of $\widetilde{G}(\A)$; 
	\item $\til{\pi}_p$ has a non-zero fixed vector under the Iwahori $\widetilde{I}$ and $\chi=(\chi_1,\chi_2)$ is an ordered pair of characters $\chi_i:\Qbar^*_p \rightarrow \C^*, i=1,2$ such that $\til{\pi}_p \hookrightarrow \operatorname{Ind}_B^{\GL_2(\Q_p)}(\chi_1,\chi_2)$, where $\operatorname{Ind}_B^{\GL_2(\Q_p)}(-)$ denotes the normalized parabolic induction from the upper triangular Borel $B \subset \GL_2(\Q_p)$; 
	\item $\til{\pi}_\infty \cong  (\operatorname{Sym}^{k_1-k_2} (\C^2) \otimes \operatorname{Nrd}^{k_2})^*$. 
\end{itemize}
Different points on an eigenvariety that come from the same automorphic representation $\til{\pi}$ are parametrized by the different choices of pairs $\chi=(\chi_1,\chi_2)$, such that $\til{\pi_p} \hookrightarrow \operatorname{Ind}_B^{\GL_2(\Q_p)}(\chi_1,\chi_2)$, which are called refinements. Note if $\pi_p \cong \operatorname{Ind}_B^{\GL_2(\Q_p)}(\chi_1,\chi_2)$, there are precisely two refinements, namely $(\chi_1,\chi_2)$ and $(\chi_2,\chi_1)$. 

Fix an idempotent $\til{e}$ and let $S:=S(\til{e})$.
Define 
\[ U_p := \mathbf{1}_{\til{\cH}_{ur,S}} \otimes \mathbf{1}_{\til{\operatorname{I}}\left(^1 \ _p\right) \til{\operatorname{I}}} \in \til{\cH}_S. 
\]
The operator 
\[  
\mathbf{1}_{\til{\cH}_{ur,S}} \otimes \mathbf{1}_{\til{\operatorname{I}}\left(^{p^{-1}} \ _p\right) \til{\operatorname{I}}} \in \til{\cH}_S.
\]
is the image of $u_0$ under $\lambda$, and we denote it by $u_0$ again. 

Recall from Section 2, that $\D(\til{e})$ comes equipped with a morphism $\psi:\til{\cH}_{S} \rightarrow \mathcal{O}(\D(\til{e}))$. 
For a point $z$ on the eigenvariety $\D(\til{e})$ which corresponds to a $p$-refined automorphic representation $(\til{\pi},(\chi_1,\chi_2))$ of weight $\underline{k}$, we have 
\[\psi(U_p)(z) = \psi_{(\til{\pi},(\chi_1,\chi_2))}(U_p) = \iota_p(\chi_2(p)p^{1/2})p^{-k_2}
\]
and 
\begin{equation}\label{u0}
\psi(u_0)(z) = \iota_p(\chi_2(p)\chi_1(p)^{-1}) p^{k_1-k_2+1}.
\end{equation}

To justify the next definition recall the following classicality theorem. Again fix~$\underline{k}$ as above and choose an affinoid neighbourhood $X$ of $\underline{k}$ in $\til{\mathcal{W}}$. We denote by $M(\til{e},\underline{k},k(X))$ the $E$-Banach space of overconvergent forms of weight $\underline{k}$ and `tame level' $\til{e}$ as defined in Section 3 of \cite{david}. It comes equipped with an action of $\til{\cH}_S$. We sometimes omit the parameter $k(X)$ and write $M(\til{e}, \underline{k})$ instead. If $x \in \D(\til{e})(\Qbar_p)$ with $\til{\omega}(x)=\underline{k}$, then there exists an overconvergent finite slope eigenform $f \in M(\til{e},\underline{k},k(X))$ with eigenvalues $\psi_x$. The space $M(\til{e},\underline{k},k(X))$ has a finite dimensional subspace $M(\widetilde{e},\underline{k})^{cl}$ of classical forms. 

\begin{thm}[{\cite[Theorem 3.9.6]{david}}] 
Let $\underline{k}=(k_1,k_2) \in \Z^2$, $k_1 \geq k_2$. Let $E'/E$ be a finite extension, $\lambda \in E'^*$ and $\sigma:= v_p(\lambda)$. If 
$$ \sigma < k_1 - k_2 +1,$$
then the generalized $\lambda$-eigenspace of $U_p$ acting on $M(\widetilde{e},\underline{k},k(X))\widehat{\otimes}_E E'$ is contained in the subspace $M(\widetilde{e},\underline{k})^{cl}\widehat{\otimes}_E E'$.  
\end{thm}

\begin{defn}
\begin{enumerate}
	\item A point $x = (\psi_x,\til{\omega}(x))$ on $\D(\til{e})$, with $\til{\omega}(x) = (k_1,k_2)$ is called \emph{of critical slope} if $v_p(\psi_x(U_p))= k_1 -k_2+1$.
  \item A refinement $\chi$ of an automorphic representation $\til{\pi}$ of $\til{G}(\A)$ of weight $(k_1,k_2)$ is called \emph{of critical slope} if $v_p(\psi_{(\til{\pi},\chi)}(U_p)) = k_1-k_2+1$. 
\end{enumerate}
\end{defn}

Now let $L/\Q$ be an imaginary quadratic extension and let $\til{\theta}:\A^*_L/L^* \rightarrow \C^*$ be a Gr{\"o}{\ss}encharacter which does not factor through the norm. 
In \cite{JL}, Jacquet and Langlands show how to associate to $\widetilde{\theta}$ a cuspidal automorphic representation $\tau(\widetilde{\theta})$ of $\GL_2(\A)$. 
We refer to \S 12 of \cite{JL} for details regarding the construction and characterization. Assume $\tau(\widetilde{\theta})$ is in the image of the global Jacquet--Langlands transfer $\operatorname{JL}$ from $\widetilde{G}$ to $\GL_2$, i.e., $\tau(\til{\theta})_v$ is a discrete series representation for all $v \in S_B$. 
Then $\pi(\til{\theta}):= \operatorname{JL}^{-1}(\tau(\til{\theta}))$ is an automorphic representation of $\til{G}(\A)$. Assume $\pi(\til{\theta})$ is of weight $(k_1,k_2) \in \Z^2$, so
\[\pi(\til{\theta})_{\infty} \cong (\operatorname{Sym}^{k_1-k_2}(\C^2) \otimes \operatorname{Nrd}^{k_2})^* \cong \operatorname{Sym}^{k_1-k_2}(\C^2) \otimes \operatorname{Nrd}^{-k_1}\]
and $\til{\theta}_{\infty}:L_{\infty}^* \rightarrow \C^*$ is given by 
$\til{\theta}_{\infty}(z) = (z\overline{z})^{-k_1-1/2}z^{k_1-k_2+1}$ (see Remark 7.7 of \cite{Gelbart}).
Let $\widetilde{e}\in C^{\infty}_c(\til{G}(\A_f^p),\Qbar)$ be an idempotent such that \ $\til{e} \cdot \pi(\til{\theta})_f^p \neq 0$. 

\begin{lem}\label{cs} Let $\pi(\til{\theta})$ be an automorphic representation of $\til{G}(\A)$ associated to a Gr{\"o}{\ss}encharacter of $L$ as above. Assume $p$ splits in $L$ and $\pi(\til{\theta})_p$ is unramified. Then ~$\pi(\til{\theta})$ has a refinement of critical slope. More precisely let $\til{x},\til{y} \in \D(\til{e})$ be the two points attached to $\pi(\til{\theta})$. Then the slopes of $\psi_{\til{x}}(U_p)$ and $\psi_{\til{y}}(U_p)$ are $k_1-k_2+1$ and~$0$. Furthermore 
\[
v_p(\psi_{\til{x}}(u_0)) = 2(k_1-k_2+1).
\]
\end{lem}
\begin{proof} 
The Gr{\"o}{\ss}encharacter $\til{\theta}_0:= \til{\theta} ||\operatorname{N}_{L/\Q}(\cdot)||^{-1/2}$ is algebraic and we can turn it into a $p$-adic character by shifting the weight from $\infty$ to $p$, i.e., we define 
\[
\til{\theta}': \A_L^*/L^* \rightarrow \Qbar_p^* 
\]
\[
\til{\theta}'(x) = \iota_p (\til{\theta}_0(x) \til{\theta}_{0,\infty}^{-1}(x_{\infty})) \tau_{w}(x_w)^{-k_2}\tau_{\overline{w}}(x_{\overline{w}})^{-k_1-1}, 
\] 
where we have matched the two complex embeddings of $L_{\infty}$ with the two places $w, \overline{w}$ above $p$.
The finite part of an algebraic Gr{\"o}{\ss}encharacter takes values in a number field. Moreover, $\til{\theta}'$ factors through the compact group $\A_L^*/L^* L_{\infty}^*$, so it takes values in $\mathcal{O}_F^*$ for some finite extension $F/\Q_p$. 

Our assumptions imply that $\pi(\til{\theta})_p \cong \operatorname{Ind}_B^{\GL_2(\Qbar_p)}(\til{\theta}_w, \til{\theta}_{\overline{w}})$ and the two refinements of $\pi(\til{\theta})$ are given by $(\til{\theta}_w,\til{\theta}_{\overline{w}})$ and $(\til{\theta}_{\overline{w}},\til{\theta}_w)$. 

Let $p_w=(1, \dotsc,1,p,1,\dotsc,1) \in \A^*_L$ (respectively $ p_{\overline{w}})$ denote the idele which is $1$ at all places except for $w$ (respectively $\overline{w}$), where it equals $p$. Then
\[\til{\theta}'(p_w) = \iota_p(\til{\theta}_{w}(p) p^{1/2})p^{-k_2} = \psi_{(\pi(\til{\theta}),(\til{\theta}_{\overline{w}}, \til{\theta}_{w}))}(U_p) \ \text{ and}\] 
\[\til{\theta}'(p_{\overline{w}}) = \iota_p(\til{\theta}_{\overline{w}}(p) p^{1/2})p^{-k_1-1} = \psi_{(\pi(\til{\theta}),(\til{\theta}_{w}, \til{\theta}_{\overline{w}}))}(U_p)p^{k_2-k_1-1}.\]
As $\til{\theta}'(p_w)$ and $\til{\theta}'(p_{\overline{w}})$ are in $\mathcal{O}_F^*$ this implies the claim on the slopes of the $U_p$-eigenvalues. Using Equation (\ref{u0}), one then verifies the slope of $\psi_{\til{x}}(u_0)$. 
\end{proof}

\begin{rem} The character $\til{\theta}':\A_L^*/L^* \rightarrow \Qbar_p^*$ in the proof of the above lemma is trivial on $L_\infty^*$, so it factors through the quotient $\A_L^*/L^*L^*_\infty\cong G_L^{ab}$. We may therefore view $\til{\theta}'$ as a continuous character of $G_L$ with values in $\Qbar_p^*$.
In this notation the Galois representation $\rho_{\pi(\til{\theta})}:G_{\Q}\rightarrow \GL_2(\Qbar_p)$ attached to $\pi(\til{\theta})$ is given by 
\[\rho_{\pi(\til{\theta})}\cong\operatorname{Ind}^{G_\Q}_{G_L}(\til{\theta}')\]
as one easily checks by comparing traces of Hecke operators and Frobenius.  
\end{rem}

\begin{lem} \label{inertslope}
Let $\pi(\til{\theta})$ be an automorphic representation of $\til{G}(\A)$ of tame level~$\til{e}$, weight $(k_1,k_2)$ and unramified at $p$ which is associated to a Gr{\"o}{\ss}encharacter $\til{\theta}:\A^*_L/L^* \rightarrow \C^*$ and assume that $p$ is inert in $L$. Then $\pi(\til{\theta})$ gives rise to two distinct points $x, y$ on $\D(\til{e})$. Their slopes agree and are equal to 
\[ v_p(\psi_x(U_p))=v_p(\psi_y(U_p))=(k_1-k_2+1)/2.  
\]\end{lem}
\begin{proof} Let $v$ denote the unique place above $p$, and let $\omega:\Q^*_p \rightarrow \C^*$ be the character associated to the quadratic extension $L_v/\Q_p$ by local class field theory. By assumption $\til{\theta}_v$ is unramified and therefore factors through the norm $\operatorname{N}_{L_v/\Q_p}:L_v^* \rightarrow \Q_p^*$. Let $\delta:\Q_p^* \rightarrow \C^*$ be a character such that $\til{\theta}_v = \delta \circ \operatorname{N}_{L_v/\Q_p}$. Then by construction $\pi(\til{\theta})_p = \operatorname{Ind}_B^{\GL_2(\Q_p)} (\delta, \delta \omega)$, in particular, the two refinements are distinct. The $U_p$-eigenvalues on the two points $x$ and $y$ are given by
\[\psi_x(U_p)= \iota_p(\delta(p)p^{1/2})p^{-k_2},\]
\[\psi_y(U_p) = \iota_p(\delta(p)\omega(p)p^{1/2})p^{-k_2}= \iota_p(-\delta(p)p^{1/2})p^{-k_2}.\]
In particular, we see that they have the same slope. A similar calculation as in the proof of the last lemma, with
\[
\til{\theta}'(x) =\iota_p (\til{\theta}_0(x) \til{\theta}_{0,\infty}^{-1}(x_{\infty})) (\operatorname{N}_{L_v/\Q_p}(x_v))^{-k_1-1}\tau_{v}(x_v)^{k_1-k_2+1}, 
\] 
implies that the slope is given by $(k_1-k_2+1)/2$.
\end{proof}

\section{Existence of $L$-indistinguishable forms}

Let $q \geq 5 $ be a prime number such that $-q \equiv 1\mod 4$ and let $L:= \Q(\sqrt{-q})$ be the associated imaginary quadratic extension. 
Choose a prime $p$ which splits in $L$ and let $B$ be the quaternion algebra over $\Q$ such that $S_B=\{q,\infty\}$. Let $\til{G}$ and $G$ be as above. 

Let $\pi(\til{\theta})$ be an automorphic representation of $\widetilde{G}(\A)$ coming from a Gr{\"o}{\ss}encharacter 
$\til{\theta}: \A_L^*/L^* \rightarrow \C^*$ 
of $L$. Assume that 
\begin{enumerate}
	\item $\pi(\til{\theta})_l$ is unramified for all $l \neq q$.
	\item The $L$-packet $\Pi(\pi(\til{\theta})_q)= \{\tau_1,\tau_2\}$ defined by $\pi(\til{\theta})_q$ is of size two. 
	\item Precisely one of the representations 
	\[\pi_1:=\bigotimes_{l \neq q} \pi^0_l \otimes \tau_1 \otimes \pi_{\infty}, \ \pi_2:= \bigotimes_{l \neq q} \pi^0_l \otimes \tau_2 \otimes \pi_{\infty}\]
	is automorphic. As before $\pi_l^0$ denotes the unique member of the local $L$-packet $\Pi(\pi(\til{\theta})_l)$, which has a non-zero fixed vector under $\SL_2(\Z_l)$.  
\end{enumerate}
\begin{lem} Automorphic representations $\pi(\til{\theta})$ of $\widetilde{G}(\A)$ satisfying the above list of properties exist.
\end{lem}

\begin{proof}
Note that $(\prod_{w \neq \infty} \mathcal{O}^*_w \times L^*_{\infty})/ \mathcal{O}^*_L \hookrightarrow \A^*_L/L^* $ is of finite index. Furthermore, our assumptions imply that $\mathcal{O}^*_L = \{1,-1\}$ and that $q$ is the only prime that ramifies in $L$. Define a Gr\"o{\ss}encharacter $\til{\theta}:\A^*_L/L^*\rightarrow \C^*$ as follows:

Let $\til{\theta}_{\infty}: L^*_{\infty}\rightarrow \C^*$ be the character given by $\til{\theta}_{\infty}(z) \mapsto (z\overline{z})^r z^m$, where $r \in \C $ and $m \geq 2$ is an even integer, so that $\til{\theta}_{\infty}$ is trivial on $\mathcal{O}^*_L$.

Denote by $v$ the unique place of $L$ above $q$. For all $w \neq v$ let $\til{\theta}_w: \mathcal{O}_w^* \rightarrow \C^*$ be the trivial character. 

For $v$, let $\mathcal{O}^1_{v}$ be the kernel of the norm map $(\operatorname{N}_{L_{v}/\Q_q})|_{\mathcal{O}^*_{v}}$. Choose any continuous non-quadratic character 
\[\theta'_{v}:\mathcal{O}^1_{v}/\{1,-1\}\rightarrow \C^*\]
and extend it to a continuous character 
\[\til{\theta}_{v}:\mathcal{O}^*_{v} \rightarrow \C^*.\]
Now the character 
\[\prod_{w \neq \infty} \til{\theta}_w \times \til{\theta}_\infty:\left(\prod_{w \neq \infty} \mathcal{O}^*_w \times L^*_{\infty}\right)/ \mathcal{O}^*_L \rightarrow \C^*\]
is continuous. Extend it arbitrarily to a character $\til{\theta}$ of $L^*\backslash \A^*_L $. 

We verify the conditions (1)--(3) for $\pi(\til{\theta})$. 
By construction $\til{\theta}_w$ is unramified for all finite places $w$ not equal to $v$, and the local extensions $L_w/\Q_l$ are unramified for $q \neq l$, which implies (1). 
Part (2) follows from Lemma 7.1 of \cite{LL}. As we have chosen a non-quadratic character $\theta'_{v}$ in the construction of $\til{\theta}_{v}$, there exists an element $\gamma \in \mathcal{O}^1_{v}$ such that $\til{\theta}_{v}(\gamma) \neq \til{\theta}_{v}(\gamma^{-1}) = \til{\theta}_{v}(\overline{\gamma})$. The other conditions of Lemma 7.1 of ~\cite{LL} are also satisfied by construction. 

Part (3) follows from the multiplicity formulae. The representations $\pi_1$ and $\pi_2$ are of type (a) and the formula for their multiplicity is given in Proposition 7.3 of ~\cite{LL}.  
\end{proof}
\begin{rem} The reason for this slightly delicate choice of the local character at the place above $q$ in the above proof is that we are constructing $L$-packets of an inner form of $\SL_2$, which is not quasi-split. Changing a representation in a global endoscopic packet at a place where the local $L$-packet is of size two therefore not always changes the multiplicity. 
\end{rem}

Now fix an automorphic representation $\pi(\til{\theta})$ as above and such that 
\[
\til{\theta}_{\infty}(z) \mapsto (z\overline{z})^{-k_1-1/2} z^{k_1-k_2+1},
\] 
where $k_1, k_2 \in \Z$ and $k_1-k_2+1\geq 2$ is an even integer. In particular, 
\[
\pi(\til{\theta})_\infty \cong (\operatorname{Sym}^{k_1-k_2}(\C^2) \otimes \operatorname{Nrd}^{k_2})^*.
\]
We have two representations $\pi_1$ and $\pi_2$ as above and we assume that $\pi_2$ is automorphic and $\pi_1$ is not. 

The representation $\pi(\til{\theta})$ shows up in the following eigenvariety:  
for all $l \neq q$ define $\til{e}_l := e_{\GL_2(\Z_l)}$ and let $\til{e}_q$ be the special idempotent attached to the Bernstein component defined by the supercuspidal representation $\pi(\til{\theta})_q$. Define $\widetilde{e}= \otimes_l \widetilde{e}_l \in C_c^{\infty}(\widetilde{G}(\A_f^p),\Qbar)$. Let $S=S(\til{e})=\{p,q\}$ and $\til{\mathcal{H}}_S:= \til{\mathcal{H}}_{ur,S}\otimes \til{\mathcal{A}}_p$ as in Section 2. Then by construction $\pi(\til{\theta})$ gives rise to two points on the eigenvariety $\D(\til{e})$, one of which is of critical slope by Lemma \ref{cs}, which we denote again by $\til{x}$. 

Let $e_{q,1}$ (respectively $e_{q,2}$) $\in C^{\infty}_c(G(\Q_q),\Qbar)$ be the special idempotent associated with~$\tau_1$ (respectively ~$\tau_2$) and define 
\[e_1 := \bigotimes_{l \neq q,p} e_{\SL_2(\Z_l)} \otimes e_{q,1} \in C^{\infty}_c(G(\A_f^p),\Qbar) \ \text{and} \]
\[e_2 := \bigotimes_{l \neq q,p} e_{\SL_2(\Z_l)} \otimes e_{q,2} \in C^{\infty}_c(G(\A_f^p),\Qbar).\]

\begin{thm}\label{mainthm}
 There exist points $x_1 \in \D(e_1)(\Qbar_p)$ and $x_2 \in \D(e_2)(\Qbar_p)$ such that 
\[ 
(\psi_{x_1},\omega(x_1)) = (\psi_{x_2},\omega(x_2)) = (\psi_{\til{x}}|_{\mathcal{H}_S}, \mu(\omega(\til{x}))).
\]
\end{thm}

\begin{proof} By construction $\pi_2$ is an automorphic representation such that $e_2 \cdot (\pi_2)_f^p \neq 0$ and so there is a point $x_2 \in \D(e_2)(\Qbar_p)$ as claimed. 

For the existence of $x_1$ we use the $p$-adic transfer. Recall the notation of Remark~\ref{diagram}. 
We have a Zariski-dense and accumulation set $Z'$ on $\D'(\widetilde{e})$, which is in bijection with the set of pairs 
\[\{(\lambda' \circ \psi_{(\til{\pi},\chi)}, (k_1,k_2))\},\]
where $\til{\pi}$ is a $p$-refined automorphic representation of $\til{G}(\A)$ of weight $(k_1,k_2)$ (cf.\ Section 3.3 of \cite{p-adicLL}). In particular, $\lambda'(\til{x}) \in Z'$. 

Let $\Pi_s$ be the set of all stable $L$-packets of $G(\A)$ and let 
\[ Z'_s:= \{ (\lambda' \circ \psi_{(\til{\pi},\chi)}, (k_1,k_2)) \in Z' \ | \ \Pi(\til{\pi}) \in \Pi_s \}\]
be the subset of $Z'$ arising from representations $\til{\pi}$ that do not come from a Gr{\"o}{\ss}en-character. This is well-defined (cf.\ \cite{p-adicLL} Section 3.3.1). 

\textbf{Claim}: There exists an open affinoid neighbourhood $U$ of $\lambda'(\til{x}) \in \D'(\til{e})(\Qbar_p)$ such that $Z'_s \cap U$ is Zariski-dense and accumulation in $U$.
Indeed choose any open affinoid neighbourhood $W$ of $\lambda'(\til{x})$ and let  
\[
V:=\{x \in W | \ v_p(\psi_x(u_0)) = v_p(\psi_{\lambda'(\til{x})}(u_0))=2(k_1-k_2+1)\}.
\]
This is an affinoid neighbourhood of $\lambda'(\til{x})$. Choose $U \subset V$ to be an open affinoid neighbourhood of $\lambda'(\til{x})$ with the property that $\omega'(U) \subset \til{\mathcal{W}}$ is open affinoid and the induced morphism $\omega'|_U: U \rightarrow \omega'(U)$ is finite and surjective when restricted to any irreducible component of $U$. 

To see that $Z'_s \cap U$ is Zariski-dense and accumulation, let $y:= \mu (\til{\omega}(\til{x}))=k_1-k_2$ and $y':= 2(k_1 - k_2) +1 \in \mathcal{W}(E)$ and define $Y:=(\mu \circ \omega')|_U^{-1}(\{y,y'\})$ to be the fibre, which is a Zariski-closed subspace of $U$ of codimension 1. 
Let $U':= U \backslash Y $. Then $Z' \cap U' \subset Z'_s$ by Lemma \ref{cs} and Lemma \ref{inertslope} above. But $Z' \cap U'$ is still Zariski-dense and accumulation in $U$, as we have only removed a Zariski-closed subset of smaller dimension. This proves the claim.

Define $e:= \bigotimes_{l \neq q,p} e_{\SL_2(\Z_l)} \otimes e_{q} \in C^{\infty}_c(G(\A_f^p),\Qbar)$, where $e_{q}$ is the special idempotent associated with the two Bernstein components defined by the representations ~$\tau_1$ and ~$\tau_2$.  
Then $\til{e}$ and $e$ are Langlands compatible and we have a $p$-adic transfer as in Theorem \ref{p-adictransfer}.

By Remark \ref{dividemp} we have a closed immersion $\D(e_1) \hookrightarrow \D(e)$, which we base-change along $\mu:\til{\mathcal{W}}\rightarrow \mathcal{W}$ to $\iota:\D''(e_1) \hookrightarrow \D''(e)$. Consider the following diagram (cf.\ Remark \ref{diagram})

\[
\xymatrix{ & & \D''(e_1) \ar[d]^{\iota} \ar[r] & \D(e_1) \ar[d] \\
\D(\til{e}) \ar[r]^{\lambda'} & \D'(\til{e})  \ar[r]^{\xi} &  \D''(e) \ar[r] &\D(e)}.
\]

As $\xi$ and $\iota$ are closed immersions of equi-dimensional rigid analytic spaces, their images are a union of irreducible components of $\D''(e)$. 
We identify $\D'(e)$ and $\D''(e_1)$  with their images in $\D''(e)$, i.e., we consider them as subspaces of $\D''(e)$.

Let $T$ be an irreducible component of $\D'(\til{e})$ containing $\lambda'(\til{x})$. Let $U$ be as above. Then $U \cap Z_s' \cap T$ is still Zariski-dense in $U \cap T$ and so there exists a point $s \in Z'_s \cap T$ and we can also assume that $s \notin T' \cap T$ for any irreducible component $T' \neq T$. As $s \in Z'_s$, $s$ comes from an automorphic representation that gives rise to a stable $L$-packet for $G$, which implies that $s \in \D''(e_1)$. Therefore $T \subset \D''(e_1)$ and in particular, $\lambda'(\til{x}) \in \D''(e_1)$.
\end{proof}

\begin{rem}
One can of course cook up other examples by using more general imaginary quadratic fields $L$ and quaternion algebras $B$ with more ramified primes.
For example assume that $B$ is ramified at more places and $\pi'$ is a representation in an endoscopic $L$-packet $\Pi(\til{\pi})$, such that $\pi'_l$ is unramified for all $l \notin S_B$, with $m(\pi')=0$, and such that $\Pi(\til{\pi}_p)$ has size one. Then one can again use the special idempotents at the bad places to construct idempotents $e \in C^{\infty}_c(G(\A_f^p),\Qbar)$ such that 
\[ 
e \cdot \pi_f^p \neq 0 \text{ for } \pi \in \Pi(\til{\pi}) \text{ if and only if } \pi = \pi'.
\]
In fact this trick works as long as $\pi'_l$ is supercuspidal at all places where it is not unramified.  
\end{rem}

\begin{corollary} In the notation of Theorem \ref{mainthm} define $\varphi:= \psi_{\til{x}}|_{\cH_S}$ and let $n=\mu(\til{\omega}(\til{x})) = k_1-k_2 \in \mathcal{W}(E)$. The eigenspaces $M(e_1,n)^{\varphi}$ and $M(e_2,n)^{\varphi}$ are both non-zero. The Galois representations attached to the eigenforms in these two spaces agree. 
\end{corollary}
\begin{proof} The Galois representations exist by Lemma \ref{galoisreps} and depend only on the points on the eigenvariety defined by the eigenforms. But the images of $x_1$ and $x_2$ in $\D(e)$, with $e$ as in the proof of Theorem \ref{mainthm}, agree by construction. 
\end{proof}

\section{Consequences}

\begin{defn}\label{clpt}
Let $\omega:\D(e)\rightarrow \mathcal{W}$ (respectively $\til{\omega}:\D(\til{e})\rightarrow \til{\mathcal{W}})$ be an eigenvariety of idempotent type $e$ (respectively $\til{e})$. We call a point $z \in \D(e)(\Qbar_p)$ (resp.\ $\D(\til{e})(\Qbar_p)$) \emph{classical} if there exists $f \in M(e,\omega(z))^{cl}$ (resp.\ $M(\til{e},\til{\omega}(z))^{cl}$) such that $h \cdot f = \psi_z(h) f$ for all $h \in \cH_{S(e)} $ (resp.\ $\til{\cH}_{S(\til{e})})$.  
\end{defn}
For the group $\til{G}$ we have the following phenomenon.
\begin{prop}\label{gl2} Assume $\til{e}' $ and $ \til{e}$ in $C^{\infty}_c(\til{G}(\A_f),\Qbar_p)$ are idempotents with $S(\til{e}')=S(\til{e})=:S$ and assume for all $l \in S$ the local idempotents $\til{e}'_l$ and $\til{e}_l$ are special idempotents associated to Bernstein components. Assume that $\til{e}*\til{e}'=\til{e}'=\til{e}'*\til{e}$ so that we have a closed immersion $h:\D(\til{e}') \hookrightarrow \D(\til{e})$. Assume $z \in \D(\til{e}')(\Qbar_p)$ is such that $h(z)$ is classical. Then $z$ is classical.
\end{prop}
\begin{proof}
We have to show that there exists a classical automorphic eigenform $f \in M(\til{e}',\til{\omega}(z))^{cl}$ with system of Hecke eigenvalues $\psi_z$. By construction of the eigenvarieties there exists an overconvergent eigenform $f_{oc} \in M(\til{e}', \til{\omega}(z))$ with system of Hecke eigenvalues given by $\psi_z$. By assumption the point $h(z)$ is classical, so there exists an eigenform $g \in M(\til{e},\til{\omega}(z))^{cl}$ for the same system of Hecke eigenvalues $\psi_{h(z)}=\psi_z$. In particular, we have a $p$-refined automorphic representation $\til{\pi}_{h(z)}$ giving rise to $h(z)$.
Both eigenvarieties $\D(\til{e}')$ and $\D(\til{e})$ carry pseudo-representations $T'$ and $T$ (see Prop.\ 3.10 of \cite{p-adicLL}) and $T'= h^* \circ T$, where $h^*: \mathcal{O}(\D(\til{e})) \rightarrow \mathcal{O}(\D(\til{e}'))$ denotes the homomorphism induced by $h$. For $l \in S$ let $I_l$ be the inertia subgroup of $\Gal(\Qbar_l/\Q)$.
By Lemma 7.8.18 of \cite{BC}, $T|_{I_l}$ is constant on the connected components of $\D(\til{e})$. Let $x \in \D(\til{e}')$ be a classical point on the same connected component as $z$. 
Then $T'_x|_{I_l}= T'_{h(z)}|_{I_l} = T_z|_{I_l}$. 
Local-global compatibility, compatibility of the local Jacquet--Langlands transfer with twists and the inertial local Langlands correspondence (see Appendix 1.2 of \cite{BMH}) imply that the local components ${(\til{\pi}_{h(z)})}_l$ and ${(\til{\pi}_x)}_l$ are in the same Bernstein component. Therefore $\til{e}'_l \cdot {(\til{\pi}_{h(z)})}_l \neq 0$. 
\end{proof}

The situation for eigenvarieties of the group $G$ is different. First of all we have the following result:
\begin{prop} In the notation of Theorem \ref{mainthm}, the point $x_1 \in \D(e_1)(\Qbar_p)$ is not classical.
\end{prop}
\begin{proof} Assume $x_1$ is classical. Then there exists an automorphic representation $\pi$ of $G(\A)$ such that 
\begin{equation} \label{e_1}
e_1 \cdot \pi^p_f\neq 0, \end{equation} 
in particular $\pi_l^{\SL_2(\Z_l)}\neq 0$ for all $l \notin \{p,q\}$. The system of Hecke eigenvalues $\psi_{\til{x}}|_{\mathcal{H}_{ur,S(\til{e})}}$ determines the representation $\pi_l$ with $\pi_l^{\SL_2(\Z_l)}\neq 0 $ and the local $L$-packet $\Pi_l=\Pi(\pi(\til{\theta})_l)$ uniquely. But this implies $\pi \in \Pi(\pi(\til{\theta}))$ (cf.\ Theorem 4.1.2 of \cite{Ramakrishnan}). Condition (\ref{e_1}) implies that $\pi_q = \tau_1$, so the only choice left might be at $p$. But by construction $\Pi(\pi(\til{\theta})_p)= \{\pi_{1,p}\}$ is a singleton and therefore $\pi \cong \pi_1$. But $m(\pi_1)=0$.          
\end{proof}

\begin{rem} It is obvious that one cannot produce non-classical points starting from Gr\"o{\ss}encharacters of an imaginary quadratic field $L$ in which $p$ is inert. There are multiple reasons for this. For example, note that any `candidate for $x_1$' that one would end up constructing would automatically be classical by the classicality theorem. The critical slope for $u_0$ is given by $2(k_1-k_2+1)$ and Lemma \ref{inertslope} implies that the slope of the $u_0$-eigenvalue of any candidate is $k_1-k_2+1$.  
\end{rem}

\begin{corollary} Let $\D(e)$ be an eigenvariety of idempotent type $e$ for $G$ and assume $z \in \D(e)(\Qbar_p)$ is a point whose system of Hecke eigenvalues 
$\psi_z$ comes from a classical automorphic representation $\til{\pi}$ of $\til{G}(\A)$. Then $z$ is not necessarily classical.
\end{corollary}

\begin{rem} Note however that one can always enforce classicality by passing to a suitable idempotent (e.g., the idempotent attached to a sufficiently small compact open subgroup $K\subset G(\A_f^p)$). In our example, the image of the non-classical point~$x_1$ under the map $\D(e_1) \hookrightarrow \D(e)$ from the proof of Theorem \ref{mainthm} is classical. The analogue of Proposition \ref{gl2} for $G$ is therefore false.
\end{rem}

\bibliography{packets}
\bibliographystyle{plain}

\end{document}